\documentclass[12pt,a4]{amsart}
\usepackage[a4paper, left=28mm, right=28mm, top=26mm, bottom=31mm]{geometry}\usepackage{amsthm,amssymb}
\usepackage{amsmath}
\usepackage{textcomp}
\usepackage{cases}
\usepackage{here}
\usepackage{graphicx}
\usepackage{amscd}
\usepackage{mathrsfs}
\usepackage{float}
\usepackage{wrapfig}
\newtheorem{thm}{Theorem}[section]
\newtheorem{dfn}[thm]{Definition}
\newtheorem{lem}[thm]{Lemma}
\newtheorem{prp}[thm]{Proposition}
\newtheorem{rem}[thm]{Remark}
\newtheorem{cor}[thm]{Corollary}

\def\P{\mathbb{P}}

\def\G{\mathbb{G}}
\def\R{\mathbb{R}}
\def\Z{\mathbb{Z}}

\def\O{\mathcal{O}}
\def\I{\mathcal{I}}

\def\Q{\mathbb{Q}}

\def\Spec{\mathop{\mathrm{Spec}}\nolimits}

\def\Hom{\mathop{\mathrm{Hom}}\nolimits}

\def\val{\mathop{\mathrm{val}}\nolimits}
\def\id{\mathop{\mathrm{id}}\nolimits}
\def\Trop{\mathop{\mathrm{Trop}}\nolimits}
\def\ini{\mathop{\mathrm{in}}\nolimits}
\def\alg{\mathop{\mathrm{alg}}\nolimits}
\def\Gr{\mathop{\mathrm{Gr}}\nolimits}
\newcommand{\relmiddle}[1]{\mathrel{}\middle#1\mathrel{}}

\numberwithin{equation}{section}
\numberwithin{figure}{section}

\makeatletter
\DeclareRobustCommand{\genericinterval}[2]{%
 \@ifstar{\genericinterval@star{#1}{#2}}{\genericinterval@nostar{#1}{#2}}}
\newcommand{\genericinterval@star}[4]{\mathopen{}\mathclose{\left#1#3,#4\right#2}}
\newcommand{\genericinterval@nostar}[4]{\mathopen{#1}#3,#4\mathclose{#2}}

\makeatother

\begin{document}
\title[A tropical characterization of non-archimedean algebraic varieties]{A Tropical characterization of algebraic subvarieties of   toric varieties over non-archimedean fields}
\author{Ryota Mikami}
\address{Department of Mathematics, Faculty of Science, Kyoto University, Kyoto 606-8502, Japan}
\email{ryo-mkm@math.kyoto-u.ac.jp}
\subjclass[2010]{Primary 14T05; Secondary 14G22; Tertiary 14M25}
\keywords{tropical geometry, rigid analytic geometry}
\date{\today}
\maketitle

\begin{abstract} 
		We study  the tropicalizations of analytic subvarieties of normal toric varieties over complete non-archimedean valuation fields.
		We show  that a Zariski closed analytic subvariety of a normal toric variety is algebraic  if its  tropicalization   is a finite union of polyhedra.
		Previously, the converse direction was known by the theorem of Bieri and Groves.
		Over the field of complex numbers, Madani,  L.¥ Nisse, and  M.¥ Nisse  proved  similar results for  analytic subvarieties of tori.		
\end{abstract}

\section{Introduction}
We study  the tropicalizations of analytic subvarieties of normal toric varieties over complete non-archimedean valuation fields.
We shall give a characterization of  algebraic subvarieties of  normal toric varieties  in terms of  their tropicalizations.

First, we recall the  definition of  the \textit{tropicalization} of Zariski closed analytic subvarieties of normal toric varieties; see \cite[Section 2 and Section 3]{Pay09-1} for details. Let $K$ be a complete non-archimedean valuation field with non-trivial absolute value $|  \cdot   |$.
Let $M $ be a free $\Z$-module of finite rank.
Let $Y_{\Sigma}$ be the normal toric variety over $K$ associated to a fan $\Sigma $ in $N_\R:=\Hom_\Z( M,\R)$.
For each cone $\sigma \in \Sigma$, the torus orbit $O(\sigma)$  corresponding to $\sigma$ is isomorphic to the torus $\Spec(K[\sigma^{\perp} \cap M]).$
The \textit{tropicalization map} $$\Trop \colon O(\sigma)^{\mathrm{an}} \rightarrow (N_\sigma )_\R:=\Hom_\Z(\sigma^{\perp} \cap M,\R)$$ 
 is the proper surjective continuous map given by 
 $$ \Trop(| \cdot |_x):= -\log |\cdot|_x  \colon \sigma^{\perp} \cap M \rightarrow  \R $$
 for $ | \cdot |_x \in O(\sigma)^{\mathrm{an}}$.
We define the tropicalization map $$\Trop \colon Y_{\Sigma}^{\mathrm{an}}=\bigsqcup_{\sigma \in \Sigma }  O(\sigma)^{\mathrm{an}} \rightarrow \bigsqcup_{\sigma \in \Sigma } (N_\sigma )_\R$$
by gluing the tropicalization maps $\Trop \colon O(\sigma)^{\mathrm{an}} \rightarrow (N_\sigma )_\R$ together; see Section 2 for details.
Here, for an algebraic variety $Z$ over  $K$, we denote by $Z^{\mathrm{an}}$ the \textit{Berkovich analytic space} associated to $Z$; see \cite[Theorem 3.4.1]{Ber90}.
For an  irreducible  Zariski closed analytic subvariety $X \subset Y_{\Sigma}^{\mathrm{an}}$, the image $\Trop(X)$ of $X$  is called the \textit{tropicalization} of $X$.

 Gubler showed that for any cone $\sigma \in \Sigma $, the intersection $\Trop(X) \cap (N_\sigma )_\R$ is a \textit{locally finite} union of polyhedra \cite[Theorem 1.1]{Gub07}.
 Here, we identify $(N_\sigma )_\R$ with $\R^{ \dim O(\sigma)}$ by taking a $\Z $-basis of $\Hom_\Z(\sigma^{\perp} \cap M,\Z)$.
On the other hand, if $X \subset Y_{\Sigma}^{\mathrm{an}}$ is  the analytification of a Zariski closed algebraic subvariety of $Y_{\Sigma}$, Bieri and Groves showed that the intersection $\Trop(X) \cap (N_\sigma )_\R$  is  a \textit{finite} union of polyhedra for any cone $\sigma \in \Sigma $ \cite[Theorem A]{BG84}. (See also \cite[Theorem 2.2.3]{EKL06}.)

 For   an irreducible Zariski closed analytic subvariety $X \subset Y_{\Sigma}^{\mathrm{an}}$,
 there is a unique cone $\sigma_X \in \Sigma $  such that $X \cap O (\sigma_X)^{\mathrm{an}}$ is a dense Zariski open analytic subvariety of $X$.

The main theorem of this paper is as follows: 

\begin{thm}\label{main}
	Let $Y_{\Sigma}$, $X \subset Y_{\Sigma}^{\mathrm{an}}$, and  $\sigma_X \in \Sigma $ be as above.
Assume that $\Trop(X ) \cap (N_{\sigma_X} )_\R$ is a finite union of polyhedra.
Then $X$ is  the analytification of a Zariski closed algebraic subvariety  of $Y_{\Sigma}$.
	\end{thm}

Conversely, if $X$ is algebraic then $\Trop(X )\cap (N_\sigma )_\R$ is a finite union of polyhedra in $(N_\sigma )_\R$ for any cone $\sigma \in \Sigma$ by the theorem of Bieri-Groves \cite[Theorem A]{BG84}.
Hence, we get the following:
\begin{cor}
		Let $Y_{\Sigma}$ and $X \subset Y_{\Sigma}^{\mathrm{an}}$ be as above.
	Then  $X$ is the analytification of a Zariski closed algebraic subvariety of $Y_{\Sigma}$ if and only if $\Trop(X )\cap (N_\sigma )_\R$ is a finite union of polyhedra in $(N_\sigma )_\R$ for any cone $\sigma \in \Sigma$.
\end{cor}

 Chow's theorem over nonarchimedean fields follows from Theorem \ref{main} as follows.
When $Y_{\Sigma}$ is the projective space, the analytic subvariety  $X$ is compact.
By \cite[Theorem 1]{Mar15}, one can see that the tropicalization $\Trop(X)$ is a  finite union of polyhedra. 
Hence, by Theorem \ref{main}, the analytic subvariety $X$ is  the analytification of a Zariski closed algebraic subvariety  of $Y_{\Sigma}$.

 In \cite{MNN14}, Madani,  L.\ Nisse, and  M.\ Nisse proved similar results for  analytic subvarieties of  tori  over the field of complex numbers.
	
	This paper is organized as follows.
	In Section 2, we   recall  basic notions of Berkovich analytic geometry.
	In Section 3, we   recall  basic notions of toric geometry and tropical geometry. We also recall that the tropicalizations of  Zariski closed analytic subvarieties of tori can be calculated using initial forms.
	In Section 4, we prove Theorem \ref{main} for analytic hypersurfaces in tori by  calculating their tropicalizations explicitly in terms of initial forms of analytic functions.
In Section 5, we prove Theorem \ref{main} by using Payne's idea in \cite{Pay09-2}. We use  surjective homomorphisms of tori to reduce to the  case of hypersurfaces.
%This  is   Payne's idea used  in \cite{Pay09-2}.
Finally, in Section 6, we give examples of the tropicalizations of  analytic and algebraic hypersurfaces in the $2$-dimensional torus.

\section{Preliminaries on Berkovich spaces}\label{sec2}
In this paper, we use the language of Berkovich analytic geometry.
 We refer to \cite{Ber90} and \cite{Ber93} for basic notations on Berkovich analytic geometry.

Let $K$ be a complete non-archimedean valuation field with non-trivial absolute value $| \cdot |$, and $K^{\circ}$ (resp.\ $k$) its valuation ring (resp.\ residue field). 
Let $\overline{\ \cdot \ } \colon K^{\circ} \rightarrow k $ be the projection.
We fix an algebraic closure $K^{\alg} $ of $K$. We also denote the extension of the  absolute value $| \cdot |$ on $K$ to $K^{\alg}$ by the same symbol.
We put $\val(a)
:=-\log |a| $ for $ a \in (K^{\alg})^\times$,  $\val(0):= \infty$, and  $\Gamma := \val ((K^{\alg})^\times)$.
There exists a group homomorphism $ \varphi \colon \Gamma \rightarrow (K^{\alg})^\times $ such that $\val \circ  \varphi  = \id_\Gamma$, where $\id_\Gamma$ is the identity map on $\Gamma$ \cite[Lemma 2.1.15]{MS15}. We fix such  $ \varphi \colon \Gamma \rightarrow (K^{\alg})^\times $ in this paper.

In this paper,  \textit{analytic spaces} mean Berkovich analytic spaces; see \cite{Ber90} and \cite{Ber93}.
For a $K$-analytic space $Z$ and a coherent ideal sheaf $\I \subset \O_Z$,  the Zariski closed $K$-analytic subspace corresponding to $\I$ is denoted by $V(\I) \subset Z$; see \cite[Proposition 3.1.4]{Ber90}.

For each scheme $Z$ locally of finite type over $K$,  the $K$-analytic space associated to $Z$ is denoted by  $Z^{\mathrm{an}}$; see \cite[Theorem 3.4.1]{Ber90}.
The scheme $Z$ is reduced, pure  $d$-dimensional,  irreducible, or separated  if and only if $Z^{\mathrm{an}}$ has the same property; see \cite[Proposition 2.7.16]{Duc17} for irreducibility, \cite[Proposition 3.4.6]{Ber90} for  separatedness, and \cite[Proposition 3.4.3]{Ber90} for  the others.
 
In this paper, an algebraic variety over $K$ means  a reduced separated scheme  of finite type over $K$.  
For an algebraic variety $Z$ over $K$, a Zariski closed analytic subvariety of $Z^{\mathrm{an}}$  means  a reduced Zariski closed $K$-analytic subspace of $Z^{\mathrm{an}}$.
We say that a Zariski closed analytic subvariety  of  $W \subset Z^{\mathrm{an}}$  is \textit{algebraic} if $W$ is  the analytification of a Zariski closed algebraic subvariety of $Z$.

We recall basic properties of separatedness and relative boundaries of morphisms of Berkovich analytic spaces; see \cite[Definition 2.5.7 and Section 3.1]{Ber90} for the definition of relative boundaries.                                                      

\begin{lem}\label{bd}
	Let $W$ be a Zariski closed analytic subvariety of the analytification $Z^{\mathrm{an}}$ of an algebraic variety $Z$ over $K$.
	Then for any $K$-analytic space $U$, a morphism $\phi \colon W \rightarrow U$ of $K$-analytic spaces is separated, and the relative boundary of $\phi$ is empty.
\end{lem}
\begin{proof}
	First, we show that $\phi$ is separated.
	Since $Z$ is separated, the analytification $Z^{\mathrm{an}}$  is separated by \cite[Proposition 3.4.6]{Ber90}.
	The Zariski closed analytic subvariety $W \subset Z^{\mathrm{an}}$ is separated by \cite[Proposition 3.1.5]{Ber90}.
	 Hence, $\phi$  is separated by \cite[Proposition 3.1.5]{Ber90}.
	 
	 Second, we show that the relative boundary of $\phi$ is empty.
	 The boundary of $Z^{\mathrm{an}}$ is empty  by \cite[Theorem 3.4.1]{Ber90}. (We note that, in \cite{Ber90}, a  $K$-analytic space is said to be closed if its boundary is empty; see  \cite[Section 3.1, p.49]{Ber90}.)
	 Since the closed immersion $W \hookrightarrow Z^{\mathrm{an}}$ has no boundary by \cite[Corollary 2.5.13 (i) and Proposition 3.1.4 (i)]{Ber90}, the boundary of $W$ is empty by \cite[Proposition 3.1.3 (ii)]{Ber90}.
   Hence the relative boundary of $\phi$ is empty by \cite[Proposition 3.1.3 (ii)]{Ber90}.
\end{proof}

\section{Tropicalizations of analytic varieties  and initial forms}

We recall some properties of normal toric varieties and their \textit{tropicalizations}; see \cite[Chapter 3]{CLS11} for toric varieties and \cite[Section 2 and Section 3]{Pay09-1} for tropicalization.
Let $M $ be a free $\Z$-module of finite rank, and  $Y_{\Sigma}$ 	 the normal toric variety over $K$ associated to a fan $\Sigma$ in $N_\R:=\Hom_\Z(M,\R)$.
There is a natural bijection between the cones $\sigma$ in $\Sigma$ and the torus orbits $O(\sigma)$ in $Y_{\Sigma}$.
For each cone $\sigma \in \Sigma$, the torus orbit $O(\sigma)$ is isomorphic to the torus $ \Spec(K[\sigma^{\perp} \cap M]).$
Its Zariski closure $\overline{O(\sigma)}$ in $Y_{\Sigma}$ is a normal toric variety over $K$ containing it as  the open dense torus orbit.
The \textit{tropicalization map} $$\Trop \colon O(\sigma)^{\mathrm{an}} \rightarrow (N_\sigma )_\R:=\Hom_\Z(\sigma^{\perp} \cap M,\R)$$ 
is the  map given by 
$$ \Trop(| \cdot |_x):= -\log |\cdot|_x  \colon \sigma^{\perp} \cap M \rightarrow  \R $$
for $ | \cdot |_x \in O(\sigma)^{\mathrm{an}}$.
It is proper, surjective, and continuous; see \cite[Section 2]{Pay09-1}.
We define the tropicalization map $$\Trop \colon Y_{\Sigma}^{\mathrm{an}}=\bigsqcup_{\sigma \in \Sigma }  O(\sigma)^{\mathrm{an}} \rightarrow \bigsqcup_{\sigma \in \Sigma } (N_\sigma )_\R$$
by gluing the tropicalization maps $\Trop \colon O(\sigma)^{\mathrm{an}} \rightarrow (N_\sigma )_\R$  together.
We define a topology on the disjoint union  $\bigsqcup_{\sigma \in \Sigma } (N_\sigma )_\R$  as follows.
We extend the canonical topology on $\R$  to that on $\R \cup \{ \infty \}$ so that  $(a, \infty]$ for $a \in \R$ are a basis of neighborhoods of $\infty$.
We also extend  the addition on $\R$  to that on $\R \cup \{ \infty \}$ by $a + \infty = \infty$ for  $a \in \R \cup \{\infty\}$.
For each cone $\sigma \in \Sigma $, we put $S_\sigma := \sigma^{\vee} \cap M$, where 
$$\sigma^{\vee}:= \{ m\in M\otimes_\Z \R \mid n(m)\geq 0 \text{ for all } n \in \sigma \}.$$
Then $\R \cup \{ \infty \}$ and $S_\sigma$ are monoids.
We consider the set of monoid homomorphisms $\Hom (S_\sigma, \R \cup \{ \infty \}) $  as a topological subspace of $(\R \cup \{ \infty \})^{S_\sigma}$. 
We define a topology on $\bigsqcup_{ \substack{\tau \in \Sigma \\ \tau \preceq \sigma  }} (N_\tau )_\R$ by the canonical bijection
$$\Hom (S_\sigma, \R \cup \{ \infty \}) \cong \bigsqcup_{ \substack{\tau \in \Sigma \\ \tau \preceq \sigma  }} (N_\tau )_\R.$$
 Then we define a topology on  $\bigsqcup_{\sigma \in \Sigma } (N_\sigma )_\R$ by gluing the topological spaces $ \bigsqcup_{\substack{\tau \in \Sigma \\ \tau \preceq \sigma  }} (N_\tau )_\R$ together.
The definition of this topology on  $\bigsqcup_{\sigma \in \Sigma } (N_\sigma )_\R$ makes sense since 
for any face $\rho \preceq \sigma$, the canonical embedding of the monoid homomorphism  $\Hom (S_\rho, \R \cup \{ \infty \})$  into $ \Hom (S_\sigma, \R \cup \{ \infty \})$   induces a homeomorphism from $\Hom (S_\rho, \R \cup \{ \infty \})$ onto its image.
We note that the tropicalization map $$\Trop \colon Y_{\Sigma}^{\mathrm{an}} \rightarrow \bigsqcup_{\sigma \in \Sigma } (N_\sigma )_\R$$
is  proper, surjective, and  continuous; see \cite[Section 3]{Pay09-1}.
For an irreducible  Zariski closed analytic subvariety $X \subset Y_{\Sigma}^{\mathrm{an}}$, the image $\Trop(X)$ of $X$  is called the \textit{tropicalization} of $X$.

 In \cite{Gub07}, Gubler showed that for any cone $\sigma \in \Sigma $, the subset $\Trop(X) \cap (N_\sigma )_\R$ of $(N_\sigma )_\R$ is a locally finite union of $\Gamma$-rational polyhedra. Here, $\Gamma = \val ((K^{\alg})^\times)$.
 (See \cite[Definition 2.3.2]{MS15} for the definition of $\Gamma$-rational polyhedra.) Moreover, he showed that for a unique cone $\sigma_X \in \Sigma $  such that $X \cap O (\sigma_X)^{\mathrm{an}}$ is a dense Zariski open analytic subvariety of $X$, the subset $\Trop(X) \cap (N_{\sigma_X} )_\R$ of $(N_{\sigma_X}  )_\R$ is a locally finite union of $d$-dimensional  polyhedra, where $d$ is the dimension of $X$ \cite[Theorem 1.1]{Gub07}.
 Here, we identify $(N_\sigma )_\R$ with $\R^{ \dim O(\sigma)}$ by taking a $\Z$-basis of $\Hom_\Z(\sigma^{\perp} \cap M,\Z)$. 
When $X=Z^{\mathrm{an}} \subset Y_{\Sigma}^{\mathrm{an}}$ for a Zariski closed algebraic subvariety $Z \subset Y_{\Sigma}$, we have  $\Trop(X) =\Trop(Z) $, and $\Trop(Z)\cap (N_\sigma )_\R$ is a finite union of   polyhedra for any cone $\sigma \in \Sigma$; see \cite[Theorem 3.3.5]{MS15}. (See also \cite[Theorem A]{BG84} and \cite[Theorem 2.2.3]{EKL06}. See \cite[Definition 3.2.1 and Section 6.2]{MS15} for the definition of the tropicalizations of algebraic subvarieties of tori and toric varieties.)

Let $M_1, M_2$ be free $\Z$-modules of finite rank, and  $\phi \colon \Spec (K[M_1]) \rightarrow \Spec (K[M_2])$  the homomorphism of algebraic tori over $K$ induced by a homomorphism $M_2 \rightarrow M_1$.
We denote by $$\Trop (\phi^{\mathrm{an}})\colon  \Hom_\Z (M_1,\R) \rightarrow \Hom_\Z(M_2,\R)$$ the $\R$-linear map such that $$\Trop (\phi^{\mathrm{an}}) \circ \Trop = \Trop \circ \phi^{\mathrm{an}};$$ see \cite[Section 1]{Pay09-1}.

We shall introduce  the initial forms of   analytic functions on the analytification $(\G_m^{r})^{\mathrm{an}}$ of the $r$-dimensional torus $$\G_m^{r}:= \Spec(K[T_1^{\pm 1}, \dots ,T_r^{\pm 1}])$$  over $K$ as follows. (See \cite[Section 2.4]{MS15} for the case of Laurent polynomials.)
Let $M'$ be the free abelian group generated by $T_i \ (1\leq i\leq r)$. We identify $\Hom_\Z(M',\R)$ with $\R^r$ by sending $\phi \in \Hom_\Z(M',\R)$ to $(\phi (T_1), \dots, \phi(T_r)) \in \R^r$.
For $u=(u_1,\dots, u_r)\in \Z^r$, we put $T^u:= T_1^{u_1} \cdots T_r^{u_r}$.

For a non-zero analytic function $$f = \sum_{u\in \Z^r} a_u T^u\in \Gamma((\G_m^{r})^{\mathrm{an}}, \O)\setminus \{0\}   \qquad (a_u \in K ),$$
 let $$\Trop (f) \colon \R^r \rightarrow \R $$ be the piecewise linear function given by 
 $$\Trop (f) (w) := \min \{\, \val(a_u) + \langle w,u\rangle \mid u \in \Z^r, a_u \neq 0 \, \} \qquad (w \in \R^r), $$ where $\langle, \rangle$ is the standard inner product on $\R^r$.
 
 We note that  a formal power series 
 $$f=\sum_{u\in \Z^r} a_u T^u \in K[[T_1^{\pm 1}, \dots ,T_r^{\pm 1}]]$$ is an analytic function on $(\G_m^{r})^{\mathrm{an}}$ if and only if  for any $w \in \R^r$, 
 we have $\lim\limits_{\lvert u\rvert \rightarrow \infty} \val(a_u) + \langle w,u\rangle =\infty $, where we put $\lvert u\rvert:= \sum_{i=1}^r \lvert u_i \rvert$.
 Hence the function $\Trop (f)$ is well-defined.
 
 \begin{dfn}
 	The {\em initial form} of $f$ with respect to $w \in \R^r$ is the Laurent polynomial over the residue field $k$ defined by  $$\ini_w(f) := \sum_{ \substack{u\in \Z^r, \ a_u\neq 0 \\  \val(a_u) + \langle w,u\rangle= \Trop (f) (w) }} \overline{a_u \varphi(-\val(a_u))} T^u \in k[T_1^{\pm 1}, \dots ,T_r^{\pm 1}] .$$
 	Here,  $ \varphi \colon \Gamma \rightarrow (K^{\alg})^\times $ is the map fixed in Section \ref{sec2}.
 \end{dfn}

Since $\lim\limits_{\lvert u\rvert \rightarrow \infty} \val(a_u) + \langle w,u\rangle =\infty $,
 the Laurent polynomial $\ini_w(f) $ is well-defined.
 
 We also note that the initial form $\ini_w(f)$ depends on the choice of $ \varphi \colon \Gamma \rightarrow (K^{\alg})^\times $, but,
  in this paper, we focus only on  whether the initial form $\ini_w(f)$ is a monomial or not for each $ w \in \R^r$. This does not depend on the choice of  $ \varphi \colon \Gamma \rightarrow (K^{\alg})^\times $.

\begin{prp}
	
For the Zariski closed analytic subvariety $Z=V(\I) \subset (\G_m^{r})^{\mathrm{an}}$ corresponding to a coherent ideal sheaf $\I \subset \O_{(\G_m^{r})^{\mathrm{an}}}$, we have 
$$\Trop(Z)=\{ \,w \in  \R^r \mid \ini_w (f) \  \text{is not a monomial for any } f \in \Gamma((\G_m^{r})^{\mathrm{an}}, \I)\setminus \{0\}   \, \}.$$
Moreover, when $\I$ is generated by a non-zero analytic function $f\in \Gamma((\G_m^{r})^{\mathrm{an}}, \O)\setminus \{0\} $, 
we have 
$$\Trop(V(f)) = \{ \,w \in  \R^r \mid \ini_w (f) \  \text{is not a monomial}  \, \} .$$ 
\end{prp}
\begin{proof}
	 For each affinoid domain $U$ of $(\G_m^{r})^{\mathrm{an}}$, $ w \in \Trop(U)$, and $f  \in \Gamma(U, \I)$ with  $\Trop^{-1}(w)\subset U$, when  the initial form $\ini_w (f)$ is not a monomial,   the initial form $\ini_w (g)$ is  not a monomial for a  function $ g  \in \Gamma(U, \I)$ which is sufficiently close to $f$ in $ \Gamma(U, \I)$.
	 By \cite[Theorem 3.1.1 and Example 3.1.3]{Van75},
	 the image of $\Gamma((\G_m^{r})^{\mathrm{an}}, \I)$ in $\Gamma(U, \I)$ is dense.
	  Hence 
	the first assertion follows from  \cite[Theorem 7.8]{Rab12}.
	One can show that $$\ini_w (fg)=\ini_w (f)\ini_w (g)$$ for any $g \in \Gamma((\G_m^{r})^{\mathrm{an}}, \O) \setminus \{0\}$  and any $w\in \R^r$ in the same way as  \cite[Lemma 2.6.2 (3)]{MS15}, where the equality is proved for Laurent polynomials.
	  Hence the second assertion holds.
	 % ; see \cite[Lemma 2.6.2 (3)]{MS15} for products of  initial forms. 
	  %(In \cite[Lemma 2.6.2 (3)]{MS15}, the equality of products of  initial forms prove only for Laurrent polynomials, but the same proof works for analytic functions.)
\end{proof}
  
\section{Tropicalization of  analytic hypersurfaces in tori}
In this section, we prove  Theorem \ref{main} for analytic hypersurfaces in the $r$-dimensional torus $(\G_m^r)^{\mathrm{an}}$.

First, we  show that the tropicalization  of an analytic hypersurface in $(\G_m^r)^{\mathrm{an}}$ is the $(r-1)$-skeleton (i.e., the union of  cells  of dimension less than or equal to $r-1$) of the polyhedral complex  associated to the analytic function defining the analytic hypersurface. (See \cite[Proposition 3.1.6 and Remark 3.1.7]{MS15} for the case of algebraic hypersurfaces.) See \cite[Section 2.3]{MS15} for  the terminology of polyhedral geometry used in this paper.

	For a non-zero analytic function $$f = \sum_{u\in \Z^r} a_u T^u\in \Gamma((\G_m^{r})^{\mathrm{an}}, \O)\setminus \{0\} \qquad (a_u \in K ),$$
	we write $ \Sigma_{\Trop (f) }$ for the coarsest polyhedral complex in $\R^r$
	containing $$ \sigma_u := \{ \, w\in \R^r \mid \Trop(f)(w)= \val(a_u) + \langle w,u\rangle \, \} $$
	for every $ u \in \Z^r$ satisfying $a_u \neq 0$, where $\langle, \rangle$ is the standard inner product on $\R^r$; see \cite[Definition 2.5.5]{MS15}.
	The polyhedral complex $ \Sigma_{\Trop (f) }$ is  pure $r$-dimensional   and its support is $\R^r$.

\begin{lem}\label{n-1skel}
	For a non-zero analytic function  $$f = \sum_{u\in \Z^r} a_u T^u\in \Gamma((\G_m^{r})^{\mathrm{an}}, \O)\setminus \{0\}   \qquad (a_u \in K ),$$
	let $V(f) \subset(\G_m^r)^{\mathrm{an}}$ be the analytic hypersurface defined by $f$. 
	Then the tropicalization $\Trop(V(f))$ is the $(r-1)$-skeleton of $ \Sigma_{\Trop (f) } $, i.e., the union of  cells of $ \Sigma_{\Trop (f) } $ of dimension less than or equal to $r-1$.
\end{lem}
\begin{proof}
	One can prove this lemma in the same way as in  the case of Laurent polynomials; see \cite[Proposition 3.1.6 and Remark 3.1.7]{MS15}.
\end{proof}

We shall now prove  Theorem \ref{main} for analytic hypersurfaces in  $(\G_m^r)^{\mathrm{an}}$.

\begin{thm}\label{hyper}
	Let $ f  \in \Gamma((\G_m^{r})^{\mathrm{an}}, \O)\setminus \{0\}$ be a non-zero analytic function.
	Assume that $\Trop(V(f))$ is a finite union of polyhedra.
	Then $f$ is a Laurent polynomial.
	In particular, the Zariski closed analytic subvariety $ V(f) \subset (\G_m^r)^{\mathrm{an}}$ is the analytification of the algebraic hypersurface of $\G_m^r$ defined by $f$.
\end{thm}
\begin{proof}
	We put $$f= \sum_{u\in \Z^r} a_u T^u \in \Gamma((\G_m^{r})^{\mathrm{an}}, \O)\setminus \{0\}.$$
	By Lemma \ref{n-1skel},  the $(r-1)$-skeleton of $\Sigma_{\Trop (f) } $ is a finite union of polyhedra.
	Since the polyhedral complex $\Sigma_{\Trop (f) }$ is pure $r$-dimensional and its support   is $\R^r$, there are only finitely many maximal cells of $ \Sigma_{\Trop (f) } $.
	We take a finite subset $\Lambda \subset \Z^r$ such that for
	each maximal cell $\sigma \in \Sigma_{\Trop (f) }$, there exists $u \in \Lambda$ satisfying $\sigma_u =\sigma$.
	Then we have $ \bigcup_{u \in \Lambda} \sigma_u = \R^r $;
	in other words, for any $w \in \R^r$, there exists $u \in \Lambda$ such that $$\Trop(f)(w)= \val(a_u) + \langle w,u\rangle. $$
	
	 We shall show that there are only finitely many $u \in \Z^r$ satisfying $a_u \neq 0$. Assume that there exist infinitely many $u \in \Z^r$ with $a_u \neq 0$. Then there exist $v=(v_1,\dots,v_r) \in \Z^r\setminus \Lambda$ and $1\leq i\leq  r$ such that $a_v \neq 0$ and 
	 $\lvert v_i\lvert > \lvert u_i\lvert $  for any $u=(u_1,\dots,u_r) \in \Lambda$.
	 Take a real number $x_i \in \R$ such that  
	 $$\val(a_v) +  x_i v_i < \val(a_u) +  x_i u_i$$
	 for any $u=(u_1,\dots,u_r) \in \Lambda$.
	 Let $x:=(0,\dots,0,x_i,0,\dots,0) \in \R^r$ be the element such that  the $i$-th entry is $x_i$ and the $j$-th entry is $0$ for $j \neq i$.
	 Then we have $$\val(a_v) + \langle x, v\rangle < \val(a_u) + \langle x,u\rangle$$
	 for any $u \in \Lambda$.
	 Hence $x\in \R^r$ is not contained in $\bigcup_{u \in \Lambda} \sigma_u$, which  contradicts   $\bigcup_{u \in \Lambda} \sigma_u =\R^r$.
	 
	 Consequently, there are only finitely many $u \in \Z^r$ satisfying $a_u \neq 0$. 
	 In other words,  $f$ is a Laurent polynomial.
\end{proof}

\begin{rem}\label{codim1}
	For an irreducible  Zariski closed analytic subvariety $Z \subset (\G_m^r)^{\mathrm{an}}$  of codimension $1$, there exists a non-zero analytic function $f \in \Gamma((\G_m^{r})^{\mathrm{an}}, \O)\setminus \{0\}$ such that $Z =V(f)$.
	This easily follows from the fact that $Z \subset (\G_m^r)^{\mathrm{an}}$ is a Cartier divisor, and the  line bundle on  $(\G_m^r)^{\mathrm{an}}$ corresponding to $Z$ is trivial; see \cite[Lemma 2.7.4]{Lut16}.
\end{rem}

\section{Proof of the main theorem}
In this section, we  shall prove Theorem \ref{main} by using  surjective homomorphisms of tori to reduce to the  case of hypersurfaces. 

Let $M $ be a free $\Z$-module of finite rank.
Let $Y_{\Sigma}$ be the normal toric variety over $K$ associated to a fan $\Sigma$ in $N_\R :=\Hom_\Z( M, \R)$, and $X$  an  irreducible Zariski closed analytic subvariety of $Y_{\Sigma}^{\mathrm{an}}$.
Take  a unique cone $\sigma_X \in \Sigma $  such that $X \cap O (\sigma_X)^{\mathrm{an}}$   is a dense Zariski open analytic subvariety of $X$. 
We fix a $\Z$-basis $m_1,\dots ,m_n$ of $(\sigma_X)^{\perp} \cap M$.
By using this $\Z$-basis, we identify $(N_{\sigma_X} )_\R:=\Hom_\Z( (\sigma_X)^{\perp} \cap M, \R)$ with $\R^n$.

Assume that $\Trop(X ) \cap (N_{\sigma_X} )_\R$ is a finite union of polyhedra.

\begin{lem}\label{orbit}
	Assume that the Zariski closed analytic subvariety $X \cap O (\sigma_X)^{\mathrm{an}}\subset O (\sigma_X)^{\mathrm{an}}$ is algebraic.
	Then the analytic subvariety $ X\subset Y_{\Sigma}^{\mathrm{an}}$ is algebraic.
\end{lem}
\begin{proof}
Let $X' \subset O (\sigma_X)$ be the  algebraic subvariety such that  $X'^{\mathrm{an}}=X \cap O (\sigma_X)^{\mathrm{an}}$.
By \cite[Proposition 3.4.4]{Ber90}, we have $(\overline{X'})^ {\mathrm{an}}=\overline{(X'^{\mathrm{an}})}=X$, where $\overline{X'}$ is the algebraic Zariski closure of $X'$ in $Y_{\Sigma}$.
 Hence the analytic subvariety $ X\subset Y_{\Sigma}^{\mathrm{an}}$ is algebraic.
\end{proof}

By Lemma \ref{orbit}, it suffices to show that $X \cap O (\sigma_X)^{\mathrm{an}}\subset O (\sigma_X)^{\mathrm{an}}$ is algebraic.

We put $X':=X \cap O (\sigma_X)^{\mathrm{an}}$.
Since $X'\subset X$ is Zariski open and $X$ is irreducible, $X'$ is irreducible.
We  put $d:= \dim X'=\dim X.$
By the $\Z$-basis $m_1,\dots ,m_n$ of $(\sigma_X)^{\perp} \cap M$, we identify $O (\sigma_X)$ with $\G_m^n$.
Let $\I \subset \O_{(\G_m^n)^{\mathrm{an}}}$ be the coherent ideal sheaf  such that $V(\I) = X' \subset (\G_m^n)^{\mathrm{an}} $.

Let $\Gr(n-d-1,n)$ be the Grassmannian of $(n-d-1)$-dimensional subspaces of $\Q^n$, which is an integral variety over $\Q$; see \cite[Section 5.3.2]{LB15}.
Let $S$ be the set  of surjective homomorphisms $$\phi \colon \G_m^n \rightarrow \G_m^{d+1}$$
 such that $\Trop(\phi^{\mathrm{an}})\colon \R^n \rightarrow \R^{d+1}$ is injective on  every $d$-dimensional polyhedron contained in $\Trop(X')$.

\begin{lem}
	The subset $$\{ \, \ker (\Trop(\phi^{\mathrm{an}})) \cap \Q^n \in \Gr(n-d-1,n)(\Q) \mid  \phi \in S \,  \} $$ is a dense Zariski open subset in $\Gr(n-d-1,n)(\Q)$.
\end{lem}
\begin{proof}

	For a $d$-dimensional $\Gamma$-rational polyhedron $P \subset \R^n$, 
	we put  $$L(P):=\{\alpha (b-a) \in \R^n \mid a,b \in P , \, \alpha \in \R \}.$$
	It is a linear subspace of $\R^n$ of dimension $d$.
	 Since the polyhedron $P$ is $\Gamma$-rational, the linear space $L(P)$ has a $\R$-basis $\{x_i \in \Q^n\}_{i=1}^d$. (See \cite[Definition 2.3.2]{MS15} for the definition of $\Gamma$-rational polyhedra.)
 
 Hence we have
 \begin{align*}
 & \{ \,A \in \Gr(n-d-1,n)(\Q) \mid \text{the projection }\R^n \rightarrow \R^n/(A\otimes \R )    \text{ is injective on } P \,  \}\\ 
  = &  \{ \,A \in \Gr(n-d-1,n)(\Q) \mid A \cap L(P)=\{0\} \,  \} \\
 = & p^{-1}(\{ \,B \in \P(\wedge^{n-d-1}\Q^n) \mid  B \wedge \wedge^{d} L(P) \neq \{0\} \, \}),
 \end{align*}
 where $$ p\colon \Gr(n-d-1,n)(\Q) \ni U  \mapsto \wedge^{n-d-1}U \in \P(\wedge^{n-d-1}\Q^n)$$ is 
  the Pl$\ddot{\mathrm{u}}$cker embedding.
 Since 
 $$\{ \,B \in \P(\wedge^{n-d-1}\Q^n) \mid  B \wedge \wedge^{d} L(P) \neq \{0\} \, \}$$ is a nonempty Zariski open subset of $ \P(\wedge^{n-d-1}\Q^n)$, 
 the subset $$\{ \,A \in \Gr(n-d-1,n)(\Q) \mid \text{ the projection }\R^n \rightarrow \R^n/(A\otimes \R )   \text{ is injective on } P\,  \}$$ is a nonempty Zariski open subset of $ \Gr(n-d-1,n)(\Q) $.
 (Remind that the algebraic variety structure of the Grassmannian $\Gr(n-d-1,n)$ is defined by the Pl$\ddot{\mathrm{u}}$cker embedding 
 $ p\colon \Gr(n-d-1,n)\rightarrow \P(\wedge^{n-d-1}\Q^n)$
 \cite[Theorem 5.2.1 and Theorem 5.2.3]{LB15}.)
 Since the Grassmannian $ \Gr(n-d-1,n) $ is irreducible, the subset $$\{ \,A \in \Gr(n-d-1,n)(\Q) \mid \text{the projection }\R^n \rightarrow \R^n/(A\otimes \R ) \text{ is injective on }P \,  \}$$ is dense in $ \Gr(n-d-1,n)(\Q) $.
	
	Since for any $A \in \Gr(n-d-1,n)(\Q)$, there exists a surjective group homomorphism $\phi \colon \G_m^n \rightarrow \G_m^{d+1}$ such that $\ker (\Trop(\phi^{\mathrm{an}})) \cap \Q^n  = A$, the subset 
\begin{align*}
& \left\{ \, \ker (\Trop(\phi^{\mathrm{an}})) \cap \Q^n \in \Gr(n-d-1,n)(\Q)  \relmiddle|
 \begin{array}{l}
 \text{a surjective group homomorphism } \phi \colon \G_m^n \rightarrow \G_m^{d+1}  \\  \Trop(\phi^{\mathrm{an}})  \colon \R^n \rightarrow \R^{d+1}  \text{ is injective on } P 
 \end{array} 
\,  \right\} \\
= &\{ \,A \in \Gr(n-d-1,n)(\Q) \mid \text{ the projection }\R^n \rightarrow \R^n/(A\otimes \R ) \text{ is injective on } P \, \}
\end{align*}
	 is a dense Zariski open subset in $\Gr(n-d-1,n)(\Q)$. (We note that for a surjective group homomorphism $\phi \colon \G_m^n \rightarrow \G_m^{d+1}$, the linear map $\Trop(\phi^{\mathrm{an}}) \colon \R^n \rightarrow \R^{d+1}$ conincides with the projection $\R^n \rightarrow \R^n/\ker (\Trop(\phi^{\mathrm{an}}))$. )
	
	Since $\Trop(X')$ is a finite union of $d$-dimensional $\Gamma$-rational polyhedra, the assertion follows.
\end{proof}

\begin{lem}\label{zariskiclosed}
	For any homomorphism $\phi \in S $,  the image $\phi^{\mathrm{an}}(X') \subset (\G_m^{d+1})^{\mathrm{an}}$ is a Zariski closed analytic subvariety.
\end{lem}
\begin{proof}
	For each $d$-dimensional polyhedron $P$ contained in  $\Trop(X') $, 
	the $\R$-linear map 
	$$\Trop(\phi^{\mathrm{an}}) |_P \colon P \rightarrow \R^{d+1} $$
	 is injective.
	Hence  $\Trop(\phi^{\mathrm{an}}) |_P$ induces a homeomorphism from $P$ onto $\Trop(\phi^{\mathrm{an}})(P)$.
	
	For any  affinoid domain $U $ of $(\G_m^{d+1})^{\mathrm{an}}$,
	since  $\Trop(U)$ is bounded, the intersection $$ \Trop(U) \cap\Trop(\phi^{\mathrm{an}})(P) $$
	is bounded.
	Hence the intersection $$ \Trop(\phi^{\mathrm{an}})^{-1}(\Trop(U)) \cap P$$
	is bounded.
	Since $\Trop(X') $ is a finite union of $d$-dimensional polyhedra, 
	the intersection $$ \Trop(\phi^{\mathrm{an}})^{-1}(\Trop(U))\cap \Trop(X') $$
	is bounded.
	Hence 
	\begin{align*}
	\Trop((\phi^{\mathrm{an}})^{-1}(U) \cap X')&\subset \Trop((\phi^{\mathrm{an}})^{-1}(U))\cap \Trop(X')\\ &\subset \Trop(\phi^{\mathrm{an}})^{-1}(\Trop(U))\cap \Trop(X')
	\end{align*}  is bounded. Thus $\Trop((\phi^{\mathrm{an}})^{-1}(U) \cap X')$ is compact.

By \cite[Lemma 2.1]{Pay09-1},  the tropicalization map 
$$\Trop \colon (\G_m^n)^{\mathrm{an}} \rightarrow \R^n$$
is proper.
	Hence the closed subset $$(\phi^{\mathrm{an}})^{-1}(U) \cap X' \subset \Trop^{-1}(\Trop((\phi^{\mathrm{an}})^{-1}(U) \cap X')) $$ is compact.
	It follows that  for any compact subset $C\subset (\G_m^{d+1})^{\mathrm{an}}$, the subset $(\phi^{\mathrm{an}})^{-1}(C) \cap X' $  is compact.
	By Lemma \ref{bd}, the morphism $\phi^{\mathrm{an}}|_{X'}\colon X' \rightarrow (\G_m^{d+1})^{\mathrm{an}}$ is separated. Hence the  map of underlying topological spaces $ |X'| \rightarrow |(\G_m^{d+1})^{\mathrm{an}}|$ induced by $\phi^{\mathrm{an}}|_{X'}$ is proper; see \cite[Section 3.1, p.50]{Ber90}.
	Moreover,  the relative boundary of $\phi^{\mathrm{an}}|_{X'}$ is empty by Lemma \ref{bd}.
	Hence  $ \phi^{\mathrm{an}}|_{X'} $ is a proper morphism of Berkovich analytic spaces; see \cite[Section 3.1, p.50]{Ber90} for the definition of proper morphisms.
	By \cite[Proposition 3.3.6]{Ber90}, the image $\phi^{\mathrm{an}}(X') \subset (\G_m^{d+1})^{\mathrm{an}}$ is a Zariski closed analytic subvariety.
\end{proof}

By Lemma \ref{zariskiclosed}, for each $\phi \in S$, we consider $\phi^{\mathrm{an}}(X') $ as a Zariski closed analytic subvariety of $ (\G_m^{d+1})^{\mathrm{an}}$.
 Since $X'$ is irreducible, its image  $\phi^{\mathrm{an}}(X') $ is also irreducible.
 Since $$\Trop(\phi^{\mathrm{an}}(X') )=\Trop(\phi^{\mathrm{an}})(\Trop(X'))$$ and $\phi \in S$, the tropicalization $\Trop(\phi^{\mathrm{an}}(X') )$ is a finite union of $d$-dimensional polyhedra.
 By \cite[Theorem 1.1]{Gub07}, the irreducible Zariski closed analytic subvariety $\phi^{\mathrm{an}}(X')\subset (\G_m^{d+1})^{\mathrm{an}}$ is  $d$-dimensional.
 By Remark \ref{codim1}, it is an analytic hypersurface in $(\G_m^{d+1})^{\mathrm{an}}$.
Hence, by Theorem \ref{hyper}, it is the analytification of an algebraic subvariety of $\G_m^{d+1}$.

 Therefore, for each $\phi \in S$, the Zariski closed analytic subvariety $$W_\phi :=(\phi^{\mathrm{an}})^{-1}(\phi^{\mathrm{an}}(X')) \subset (\G_m^n)^{\mathrm{an}}$$  is algebraic.
We put $$W:= \bigcap_{\phi \in S} W_\phi .$$
%We put $Z$ the union of irreducible component of $Z'$ of dimension greater than or equal to $d$.
Then $W$ contains $X'$.
 Since  $W_\phi \subset (\G_m^n)^{\mathrm{an}}$ is algebraic for every $\phi \in S$, the intersection $W$ is  algebraic.
 One can  deduce that  $X'\subset (\G_m^n)^{\mathrm{an}}$ is algebraic from the following:

\begin{lem}\label{dim}
 The dimension of	$W$ is less than  or equal to $d=\dim X'=\dim X$.
\end{lem}
\begin{proof}
	Assume that there exists an irreducible component $V$ of $W$ of dimension greater than $d$.
	By \cite[Theorem 1.1]{Gub07},  the tropicalization $\Trop(V) $ is a locally finite union of polyhedra of dimension greater than $d$.
	We take a polyhedron $P$ of dimension $d+1$ contained in $\Trop(V) $.
	
	Let $S'$ be  the set  of surjective homomorphisms $$\phi \colon \G_m^n \rightarrow \G_m^{d+1}$$ such that the $\R$-linear map $\Trop(\phi^{\mathrm{an}})\colon \R^n \rightarrow \R^{d+1}$ is injective on $P$.
	By  \cite[Theorem 5.2.3]{LB15}, the subset $$\{ \, \ker (\Trop(\phi^{\mathrm{an}}))\cap \Q^n  \mid  \phi \in S' \, \} $$  is dense Zariski open in $\Gr(n-d-1,n)(\Q)$.
	 Since $S$ is also Zariski  open dense in  $\Gr(n-d-1,n)(\Q)$, the intersection $S \cap S'$ is non-empty.
	 
	 We take an element $\phi \in S \cap S'$.
	 Since $\phi^{\mathrm{an}}(W_\phi) = \phi^{\mathrm{an}}(X') $, 
	 	we have $$\Trop(\phi^{\mathrm{an}})(\Trop(W_\phi))= \Trop(\phi^{\mathrm{an}}(W_\phi)) = \Trop(\phi^{\mathrm{an}}(X'))=\Trop(\phi^{\mathrm{an}})(\Trop(X')) .$$
	 Since $\phi \in S$, the image $\Trop(\phi^{\mathrm{an}})(\Trop(W_\phi))$ is a finite union of  $d$-dimensional polyhedra. 
	 	Since $\phi \in S'$, the polyhedron $\Trop(\phi^{\mathrm{an}})(P)$ is $(d+1 )$-dimensional.
	 	 Hence we have $$\Trop(\phi^{\mathrm{an}})(P)\not\subset\Trop(\phi^{\mathrm{an}})(\Trop(W_\phi)).$$
	 	 Since $P \subset \Trop(V)$,
we have $$\Trop(\phi^{\mathrm{an}})(\Trop(V)) \not\subset \Trop(\phi^{\mathrm{an}})(\Trop(W_\phi)),$$
 which contradicts $V \subset W \subset W_\phi$.
 
Consequently, the dimension of	$W$ is less than  or equal to $d$.
\end{proof}

\begin{proof}[Proof of Theorem \ref{main}]
	
	Recall that $X':=X \cap O (\sigma_X)^{\mathrm{an}}$ and $d:=\dim X'= \dim X$. 
	By Lemma \ref{dim}, the dimension of $W$ is less than  or equal to $d$.
	Since $X'\subset W$, the analytic subvariety  $X'$ is an irreducible component of $W$. 
	Since $W$ is algebraic, by \cite[Proposition 2.7.16]{Duc17}, the Zariski closed analytic subvariety $X'\subset O (\sigma_X)^{\mathrm{an}}$ is algebraic.
	Therefore, by Lemma \ref{orbit}, the analytic subvariety $X\subset Y_{\Sigma}^{\mathrm{an}}$ is algebraic.
	
	The proof of Theorem \ref{main} is complete.
	\end{proof}

 \section{Examples of the tropicalizations of algebraic and analytic  hypersurfaces}
 
 In this section, we give two examples of  the tropicalizations of   hypersurfaces in the $2$-dimensional torus $(\G_m^2)^{\mathrm{an}}$.
 The first example is analytic and not algebraic. It is a locally finite union of polyhedra, but the number of polyhedra is infinite.
 The second example is algebraic, and it is a finite union of  polyhedra.
 
  Let $$f = \sum_{(i,j)\in \Z_{\geq 0}^2 } a_{i,j} X^i Y^j\in  \O ((\G_m^2)^{\mathrm{an}})  \setminus \{0\} \qquad (a_{i,j} \in K )$$ 
 be a non-zero analytic function on $(\G_m^2)^{\mathrm{an}}=(\Spec K[X^{\pm},Y^{\pm}])^{\mathrm{an}}$
 satisfying $$\val(a_{i,j})= i^2+j^2+ij-i-j$$ for any $(i,j)\in \Z_{\geq 0}^2 $.
 For each $(s,t)\in \Z_{\geq 0}^2 $, we put  
 $$f_{s,t} := \sum_{\substack{0\leq i \leq s\\ 0\leq j \leq t}} a_{i,j} X^i Y^j\in K[X,Y]\setminus \{0\} .$$ 
 
 First, we consider the tropicalization $\Trop(V(f)) \subset \R^2$.
 Since $f$ has infinitely many non-zero terms, the analytic hypersurface $V(f)\subset (\G_m^2)^{\mathrm{an}}$ is not algebraic. The tropicalization $\Trop(V(f)) \subset \R^2$ is not a finite union of  polyhedra; see \textsc{Figure.1}.
 
 \begin{center}
 	\includegraphics{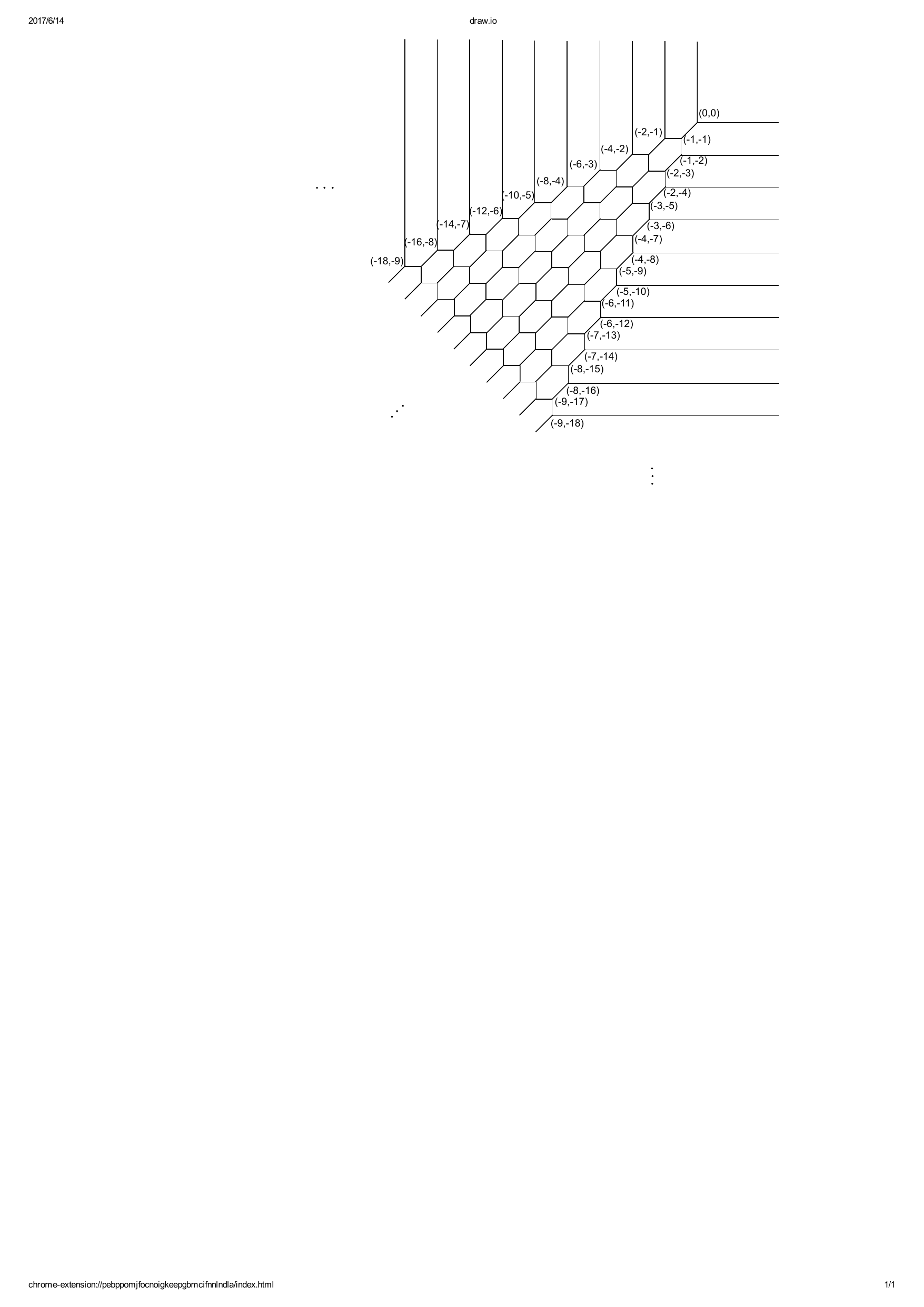}
 	
 	\textsc{Figure.1}.\ The tropicalization of the analytic hypersurface  $V(f)\subset (\G_m^2)^{\mathrm{an}}$.
 \end{center}
 
 Next, we consider the tropicalization $\Trop(V(f_{s,t})) \subset \R^2$ for each $(s,t)\in \Z_{\geq 0}^2 $.
 Since $f_{s,t}$ is a polynomial, the analytic hypersurface $V(f_{s,t})\subset (\G_m^2)^{\mathrm{an}}$ is algebraic. The tropicalization $\Trop(V(f_{s,t})) \subset \R^2$ is a finite union of polyhedra; see Figure.2.
  
  \begin{center}
	\includegraphics[width=15cm]{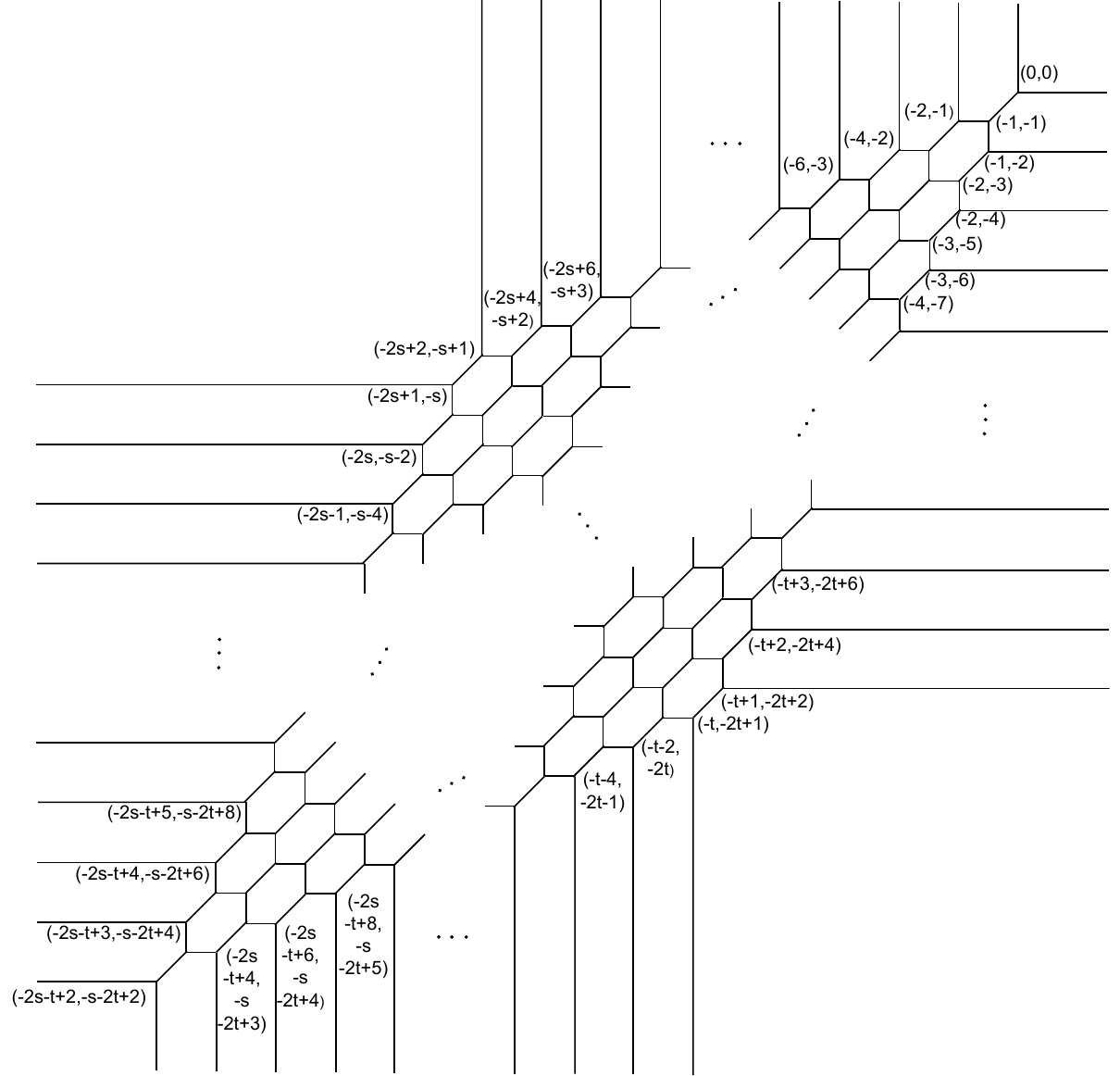}
	
    	\textsc{Figure.2}.\ The tropicalization of the algebraic hypersurface  $V(f_{s,t})\subset (\G_m^2)^{\mathrm{an}}$.
    
    	\end{center}

\subsection*{Acknowledgements}
The author would like to express his sincere thanks to his adviser Tetsushi Ito for  invaluable advice and persistent support.
He is deeply grateful to Teruhisa Koshikawa for pointing out the insufficency of the argument for irreducibility and teaching him the refference of it. This work was supported by JSPS KAKENHI Grant Number 18J21577.


\begin{thebibliography}{99}
	
	
				\bibitem[Ber90]{Ber90}
		Berkovich, V. G., \textit{Spectral theory and analytic geometry over non-Archimedean fields}, Mathematical Surveys and Monographs, 33.\ American Mathematical Society, Providence, RI, 1990.
		
		\bibitem[Ber93]{Ber93}
		Berkovich, V. G., \textit{\'{E}tale cohomology for non-Archimedean analytic spaces}, Inst.\ Hautes \'{E}tudes Sci.\ Publ.\ Math.\ No.\ \textbf{78} (1993), 5-161.
	
	
	
		\bibitem[BG84]{BG84}
		Bieri, R, Groves, J. R. J., \textit{The geometry of the set of characters induced by valuations}, J.\ Reine Angew.\ Math.\ \textbf{347} (1984), 168-195.
		
%			\bibitem[Con99]{Con99}
%		Conrad, B., \textit{Irreducible components of rigid spaces}, Ann.\ Inst.\ Fourier (Grenoble) \textbf{49} (1999), no.\ 2, 473-541.
%		
		\bibitem[CLS11]{CLS11}
		 Cox, D. A., Little, J. B., Schenck, H. K., \textit{Toric varieties}, Graduate Studies in Mathematics, 124.\ American Mathematical Society, Providence, RI, 2011.
		 
		 \bibitem[Duc17]{Duc17}
		  Ducros, A., \textit{Families of Berkovich spaces}, preprint, arXiv:1107.4259v6, to appear as a memoir in the Ast\'erisque collection
		
		\bibitem[EKL06]{EKL06}
		Einsiedler, M., Kapranov, M., Lind, D., \textit{Non-Archimedean amoebas and tropical varieties}, J.\ Reine Angew.\ Math.\ \textbf{601} (2006), 139-157. 
		
		\bibitem[Gub07]{Gub07}
		 Gubler, W., \textit{Tropical varieties for non-Archimedean analytic spaces}, Invent.\ Math.\ \textbf{169} (2007), no.\ 2, 321-376.
	
		\bibitem[Gub13]{Gub13}
		 Gubler, W., \textit{A guide to tropicalizations}, Algebraic and combinatorial aspects of tropical geometry, 125-189, Contemp.\ Math., 589, Amer.\ Math.\ Soc., Providence, RI, 2013. 
		 	
		 	\bibitem[LB15]{LB15}
		 Lakshmibai, V., Brown, J., \textit{The Grassmannian variety. Geometric and representation-theoretic aspects}, Developments in Mathematics, 42.\ Springer, New York, 2015.
		 
			\bibitem[Lut16]{Lut16} 
		L\"utkebohmert, W., \textit{Rigid geometry of curves and their Jacobians},
		Ergebnisse der Mathematik und ihrer Grenzgebiete.\ 3.\ Folge.\ A Series of Modern Surveys in Mathematics,
		\textbf{61}.\ Springer, Cham, 2016.
	
	
	\bibitem[Mar15]{Mar15}
	Martin, F., \textit{A note on tropicalization in the context of Berkovich spaces}, Israel J.\ Math.\ \textbf{210} (2015), no.\ 1, 323–334. 
	
			\bibitem[MNN14]{MNN14}
		Madani, F.,  Nisse, L., Nisse, M., \textit{A tropical characterization of complex analytic varieties to be algebraic}, preprint, arXiv:1407.6459
		
			\bibitem[MS15]{MS15}
		Maclagan, D., Sturmfels, B., \textit{Introduction to tropical geometry}, Graduate Studies in Mathematics, 161.\ Amer.\ Math.\ Soc., Providence, RI, 2015.
		
		
		\bibitem[Pay09-1]{Pay09-1}
		Payne, S., \textit{Analytification is the limit of all tropicalizations}, Math.\ Res.\ Lett.\ \textbf{16} (2009), no.\ 3, 543-556.
	
	\bibitem[Pay09-2]{Pay09-2}
		Payne, S., \textit{Fibers of tropicalization}, Math.\ Z.\ \textbf{262} (2009), no. 2, 301-311. 
		
		\bibitem[Rab12]{Rab12}
		Rabinoff, J., \textit{Tropical analytic geometry, Newton polygons, and tropical intersections}, Adv.\ Math.\ \textbf{229} (2012), no.\ 6, 3192-3255.
		
		\bibitem[SS04]{SS04}
	 Speyer, D., Sturmfels, B., \textit{The tropical Grassmannian}, Adv.\ Geom.\ \textbf{4} (2004), no.\ 3, 389-411.
		
	
	\bibitem[Van75]{Van75}
	 van der Put, M., \textit{Rigid analytic spaces}, Groupe d'\'Etude d'Analyse Ultram\'etrique (3e ann\'ee: 1975/76), Fasc.\ 2 (Marseille-Luminy, 1976), Exp.\ No.\ J7, 20 pp.\ Secr\'etariat Math., Paris, 1977.
		


\end{thebibliography}
\end{document}